\documentclass[a4paper,11pt]{amsart}

\usepackage{amsmath}       
\usepackage{amsthm}        
\usepackage{amsfonts}      
\usepackage{amssymb}       
\usepackage{mathtools}     
\usepackage{mathrsfs}       
\usepackage{dsfont}        
\usepackage{bbm}           
\usepackage{esvect}        
\usepackage{braket}        
\usepackage{cases}          
\usepackage{amscd}         

\usepackage{pifont}        
\usepackage{latexsym}      
\usepackage{savesym}       
\usepackage{ifthen}        
\usepackage[misc]{ifsym}   
\usepackage{verbatim}      

\usepackage[shortlabels,inline]{enumitem} 

\usepackage{graphicx}      
\usepackage{xcolor}        
\usepackage{tikz}          
\usepackage{tikz-cd}       
\usepackage{float}         
\usepackage{extarrows}     
\usepackage{tcolorbox}     
\usepackage{framed}        

\usepackage{geometry}      
\usepackage{indentfirst}   
\usepackage{needspace}     
\usepackage{dcolumn}       
\usepackage{titletoc}      

\usepackage[utf8]{inputenc}
\usepackage[T1]{fontenc}    

\usepackage[numbers,sort]{natbib} 
\usepackage{aliascnt}       

\usepackage[
pdfusetitle,
bookmarks=true,
pagebackref=false,      
colorlinks,
linkcolor=blue,         
citecolor=blue,         
urlcolor=black
]{hyperref}

\newcommand*{\dif}{\,\mathrm{d}}
\newcommand*{\dih}{\,\mathrm{d}h}
\newcommand{\R}{\mathbb{R}}
\newcommand{\mP}{\mathcal{P}}

\newcommand{\uid}{\overline{\dim}_\theta}
\newcommand{\lid}{\underline{\dim}_\theta}
\newcommand{\ubd}{\overline{\dim}_B}
\newcommand{\lbd}{\underline{\dim}_B}

\newcommand{\ca}{C_r^m}
\newcommand{\prm}{\phi_r^m}
\newcommand{\sk}{s_k(s,\mu,r)}
\newcommand{\ak}{\alpha_k(s,\mu,r)}

\numberwithin{equation}{section}

\allowdisplaybreaks
\theoremstyle{plain}

\newtheorem{thmm}{Theorem}[section]
\newtheorem{lem}{Lemma}[section]
\newtheorem{pro}{Proposition}[section]
\newtheorem{cor}{Corollary}[section]

\newtheorem{ex}{Example}[section]

\theoremstyle{definition}
\newtheorem{defi}{Definition}[section]
\newtheorem{remark}{Remark}[section]

\title{Box dimension profiles and intermediate dimensions}
\author{Shuqin Zhang}

\address{School of Mathematics and Information Science, Guangzhou University, Guangzhou, China}
\email{zhangshuqin@gzhu.edu.cn}

\begin{document}
\begin{abstract}
    For a non-empty bounded set $E\subset\R^d$, the upper and lower box dimensions of its orthogonal projections are constant onto almost all $m$-dimensional subspaces. These almost-sure values are called the upper and lower $m$-box dimension profiles of $E$, respectively. We study their relationship with intermediate dimensions, which interpolate between Hausdorff and box dimensions by restricting the relative sizes of covering sets. We establish several equivalent characterizations of both quantities and obtain upper and lower bounds for box dimension profiles in terms of intermediate dimensions, with the lower bounds also involving the upper Assouad spectrum. 
    We further derive explicit formulas for the box dimension profiles in terms of the box dimension and the initial growth rate of the intermediate dimensions, under a suitable condition on the quasi-Assouad dimension.
\end{abstract}
 		\maketitle
		 \section{Introduction}
Let $E$ be a subset of $\mathbb{R}^d$, and let $P_VE$ denote the orthogonal projection of $E$ onto a subspace $V\in G(d,m)$, where $G(d,m)$ is the Grassmannian of $m$-dimensional subspaces of $\mathbb{R}^d$ with $m\in\{1,\ldots,d\}$. The well-known Marstrand's projection theorem states that for every Borel set $E\subset\mathbb{R}^d$,
	 \begin{equation}\label{marstrand}
		 	\dim_H P_V(E)=\min\{m,\dim_H E\}
		 \end{equation}
		 for $\gamma_{d,m}$-almost all $V\in G(d,m)$, where $\gamma_{d,m}$ denotes the natural invariant probability measure on $G(d,m)$; see \cite{Mar54,Ma75}.
		 Marstrand's original proof of \eqref{marstrand} treated orthogonal projections of planar sets onto lines and was geometrically intricate. Kaufman \cite{Kau} later introduced a potential-theoretic approach, providing a simpler proof that has become the foundation for many subsequent generalizations. In 1975, Mattila \cite{Ma75} extended the theorem to higher dimensions.

A natural question is whether an analogue of \eqref{marstrand} holds for other notions of dimension. The corresponding formula fails for the upper and lower box dimensions, although their almost-sure constancy is retained. More precisely, for a non-empty bounded set $E\subset\mathbb{R}^d$, the upper and lower box dimensions of $P_VE$ are constant
for $\gamma_{d,m}$-almost all $V\in G(d,m)$; see
Falconer and Howroyd \cite{FH} and Howroyd \cite{Ho}. These almost-sure values are called the upper and lower $m$-box dimension profiles of $E$, denoted by $\ubd^m E$ and $\lbd^m E$, respectively. Moreover, Falconer and Howroyd \cite{FH} established that
 for Borel set $E$,
		\begin{equation}\label{bound}
			\frac{\ubd E}{1+\left(\frac1m-\frac1d\right)\ubd E}
			\le
			\ubd(P_VE)
			\le
			\min\{m,\ubd E\}
		\end{equation}
		for $\gamma_{d,m}$-almost all $V\in G(d,m)$.  The analogous inequalities also hold for the lower box dimension. J\"arvenp\"a\"a \cite{Ja} constructed an example showing that the lower bound in \eqref{bound} is sharp. 
In \cite{Fa}, Falconer  first gave alternative definitions of box dimension profiles in terms of capacities of $E$ with respect to certain kernels; see Definition~\ref{box-dimension-profiles}.

Intermediate dimensions were introduced in \cite{FFK} and have attracted considerable attention in recent years; see, for example, \cite{FFK,Ba,BK,BR,Fe,AMT,Fa22} and the survey \cite{Fa21}. The lower and upper $\theta$-intermediate dimensions, denoted by $\lid E$ and $\uid E$, interpolate between Hausdorff and box dimensions by restricting the relative sizes of sets in a cover. 
The parameter $\theta$ controls these restrictions, with $\theta=0$ giving the Hausdorff dimension and $\theta=1$ giving the corresponding box dimensions. The intermediate dimensions need not be continuous at $\theta=0$, so their right-hand limits at zero may differ from the Hausdorff dimension.
 For $1\leq m<d$, Angelini and Molter \cite[Corollary~3.4]{AM}  showed that

	\begin{align}\label{intro1}  
		\ubd^m(E)=m &\Longleftrightarrow \lim_{\theta\to0^+}\uid E\geq m,\notag\\ \lbd^m(E)=m &\Longleftrightarrow \lim_{\theta\to0^+}\lid E\geq m, \end{align}
	which	extended the result of Burrell, Falconer, and Fraser \cite[Corollary~6.4]{BFF}.

Falconer, Fraser, and Shmerkin \cite[Theorem~2.2]{FFS} obtained a lower bound for box dimension profiles involving the Assouad dimension and the upper Assouad spectrum. Their result can improve the lower bound in \eqref{bound} under suitable assumptions. In particular, they proved that if $\dim_{qA}E\leq m$, then the upper and lower box dimensions are preserved under orthogonal projections onto $\gamma_{d,m}$-almost all $m$-dimensional subspaces; that is,
 \begin{equation}\label{intro2}
			\ubd^mE=\ubd E
			\quad\text{and}\quad
			\lbd^mE=\lbd E.
		\end{equation}

		Projection theorems have also been studied for other notions of dimension. Burrell, Falconer, and Fraser \cite{BFF} showed that the intermediate dimensions of orthogonal projections are constant for almost all $m$-dimensional subspaces. They characterized these almost-sure values, called intermediate dimension profiles, using a three-piece potential kernel. For the Assouad dimension, however, this almost-sure constancy may no longer hold, as demonstrated by Fraser and Orponen \cite{FO}. Falconer, Fraser, and K\"{a}enm\"{a}ki
		\cite[Theorem~2.7 and Corollary~2.10]{FFK2}
		showed that almost-sure constancy may also fail for the
		quasi-Assouad dimension and the Assouad spectrum
		of projections of unbounded sets.

The equivalences in \eqref{intro1} characterize when a box dimension profile attains the maximal value $m$, whereas \eqref{intro2} gives a condition under which the box dimension is preserved. Alongside the existing quantitative estimates, we seek bounds that use the intermediate dimension, particularly when a profile is strictly smaller than both $m$ and the corresponding box dimension.
Since orthogonal projections do not increase distances and intermediate dimensions record how the efficiency of coverings changes with the range of admissible scales, it is natural to investigate how intermediate dimensions relate to box dimension profiles.

  In Section~\ref{sec-intermediate} and \ref{sec-profile}, we   establish equivalent characterizations of intermediate dimensions and box dimension profiles, see Proposition~\ref{alphaid} and Theorem~\ref{main}. These describe the dimensions through local mass estimates over varying ranges of scales.
In Section~\ref{sec:proof}, we obtain general upper and lower bounds for box dimension profiles in terms of intermediate dimensions. In particular, Theorem~\ref{upperbound2} shows that, for every non-empty bounded set $E\subset\R^d$ and every $m>0$,
\begin{equation*}
    \ubd^m E\leq
    \inf_{\theta\in(0,1]}
    \frac{m\uid E}{m\theta+(1-\theta)\uid E},
\end{equation*}
with an analogous bound for the lower box dimension profile. Theorem~\ref{lowerbound-spectrum} provides lower bounds that also incorporate the upper Assouad spectrum. Its consequence, Corollary~\ref{cor:assouad-spectrum-bound}, gives a lower bound at least as strong as that of Falconer, Fraser, and Shmerkin \cite{FFS}.
When  $N:=\frac{\kappa\ubd E}{\kappa-\ubd E}\geq\dim_{qA} E$, where $\kappa:=\liminf_{\theta\downarrow0}\frac{\uid E}{\theta}$,
these bounds yield an exact formula $$\ubd^m E
=\min\left\{\ubd E,\frac{m\kappa}{m+\kappa}\right\},$$ 
 see    Theorem~\ref{thm:initial-growth-profiles}.
  Finally, the inverted lattice and product examples in Examples~\ref{ex:inverted-lattice-profiles} and~\ref{ex:product-geometric-profiles} illustrate the exact formula in cases of dimension drop and dimension preservation.

  \section{Preliminaries}\label{preliminaries}
  Let $E$ be a non-empty bounded subset of $\R^d$, and
  let $\mP(E)$ denote the set of Borel probability measures supported on $E$.  We let $B(x,r)$ denote the closed ball centered at $x$ with radius $r$, and we let $|U|$ denote the diameter of the set $U$.
Let $N_r(E)$ denote the minimum number of sets of diameter at most $r$ required to cover $E$. Unless otherwise specified, all suprema \(\sup_x\) and infima \(\inf_x\) are taken over \(x\in\mathbb R^d\).  We adopt the convention $\log0=-\infty$ and $\sup\emptyset=0$.
We write $f(r)\lesssim g(r)$ if $f(r)\leq cg(r)$ for some constant $c>0$ independent of $r$.  

  \begin{defi}\label{box-dimension-profiles}
  	Let $E\subset\R^d$ be non-empty and compact, and let $m>0$ and $r>0$. The potential kernel is defined by
  	\begin{equation*}
  		\phi_r^m(x)=
  		\begin{cases}
  			1, & |x|\leq r,\\[1mm]
  			\left(\dfrac{r}{|x|}\right)^m, & |x|>r.
  		\end{cases}
  	\end{equation*}
  	The energy of $\mu\in\mP(E)$ with respect to $\prm$ is defined by  
  	\begin{equation*}
  		\int \phi_r^m\ast\mu(x)\dif\mu(x)
  		=\iint\phi_r^m(x-y)\dif\mu(y)\dif\mu(x),
  	\end{equation*}
  	where $\ast$ denotes convolution. The capacity $\ca(E)$ of $E$ is the reciprocal of the minimum energy achieved by probability measures on $E$, that is,
  	\begin{equation*}
  		\frac{1}{\ca(E)}
  		=\inf_{\mu\in\mP(E)}\int \phi_r^m\ast\mu(x)\dif\mu(x).
  	\end{equation*}
  	Since $\phi_r^m$ is continuous and $E$ is compact, the infimum is attained and $0<C_r^m(E)<\infty$. For a general non-empty bounded set $E\subset\R^d$, the capacity is defined to be that of its closure.
  	The lower and upper {\em $m$-box dimension profiles} of $E$ are defined by
\begin{equation*}
  		\lbd^m E=\varliminf_{r\rightarrow0}\frac{\log\ca(E)}{-\log r}
  		\quad\text{and}\quad
  		\ubd^m E=\varlimsup_{r\rightarrow0}\frac{\log\ca(E)}{-\log r}.
  	\end{equation*}
  \end{defi}

It was shown in \cite{Fa} that, for $m>d$, the upper and lower $m$-box dimension profiles coincide with the corresponding box dimensions of $E$. If $m\in\{1,\ldots,d-1\}$, the upper and lower box dimensions of the orthogonal projections of $E$ onto almost all $m$-dimensional subspaces coincide with the corresponding box dimension profiles.

Intermediate dimensions were introduced by Falconer, Fraser, and Kempton \cite{FFK} to interpolate between the Hausdorff and box dimensions. For $0<\theta\leq1$, we use covers satisfying $r\leq |U_i|\leq r^\theta$. This is equivalent to the original covering condition $\delta^{1/\theta}\leq |U_i|\leq\delta$ under the substitution $\delta=r^\theta$.
 \begin{defi}
 	Let $0<\theta\leq1$. The {\em lower $\theta$-intermediate dimension} of $E$  is defined by 
 	\[
 	\begin{aligned}
 		\underline{\dim}_\theta E =\inf \Big\{ s\geq 0: & \text{ for all } \epsilon > 0 \text{ and all } r_0 > 0, \text{ there exists } 0 < r \leq r_0  \text{ and }\\
 		& \text{a cover } \{U_i\} \text{ of } E \text{ such that } r  \leq |U_i| \leq r^{\theta} \text{ and } \sum |U_i|^s \leq \epsilon\Big\}.
 	\end{aligned}
 	\]
 	Similarly,  the {\em upper $\theta$-intermediate dimension} of $E$  is defined by
 	\[
 	\begin{aligned}
 		\overline{\dim}_\theta E =\inf& \Big\{ s\geq 0: \text{ for all } \epsilon > 0 \text{ there exists } r_0 > 0  \text{ such that for all  } 0 < r \leq r_0, \\
 		& \text{there is a cover } \{U_i\} \text{ of } E\text{ such that } r  \leq |U_i| \leq r^{\theta} \text{ and } \sum |U_i|^s \leq \epsilon\Big\}.
 	\end{aligned}
 	\]
 \end{defi}
  With these definitions,                                                            
  \begin{equation*}
  	\underline{\dim}_1 E = \underline{\dim}_B E   \quad
  	\text{ and }\quad \overline{\dim}_1 E= \overline{\dim}_B E.
  \end{equation*}
  It was proved in \cite{FFK} that upper and lower intermediate dimensions are non-decreasing and continuous on $(0,1]$. 
At the endpoint $\theta=0$,   $\underline{\dim}_0 E=\overline{\dim}_0 E=\dim_H E$. Continuity at $\theta=0$ need not hold.

We also recall the Assouad-type dimensions that will be used in this paper.
For $0<\theta<1$, the {\em upper Assouad spectrum} of $E$ is defined by
  	\begin{align}\label{defi-upperassouad}
  		\overline{\dim}_A^\theta E = \inf \Bigg\{ t \geq 0 : {} & \text{there exists a constant } c>0 \text{ such that for all }\notag \\
  		 & 0<r\leq R^{\frac{1}{\theta}}<R<1 \text{ and } x\in E,  
  	  N_r(B(x,R)\cap E) \leq c\left(\frac{R}{r}\right)^t \Bigg\}{.}
  	\end{align}

  	 The upper Assouad spectrum is non-decreasing in $\theta$.
  	The {\em quasi-Assouad dimension} of $E$ is  defined as the limit of the upper Assouad spectrum as $\theta$ approaches $1$, i.e.,
  	\begin{equation*}
  		\dim_{qA} E = \lim_{\theta\nearrow 1} \overline{\dim}_A^\theta E.
  	\end{equation*}
  	The {\em Assouad dimension} $\dim_A E$ is obtained from \eqref{defi-upperassouad} by replacing the scale restriction $0<r\leq R^{1/\theta}<R<1$ with $0<r<R<1$.
  	
  	 Generally, for all $\theta \in (0, 1)$,
 \[
   \underline{\dim}_{\mathrm{B}}E\leq  \overline{\dim}_{\mathrm{B}} E\leq \overline{\dim}_{\mathrm{A}}^\theta E \leq \dim_{\mathrm{qA}} E \leq \dim_{\mathrm{A}} E{.}
 \]

For each $0<\theta\leq1$, the upper and lower $\theta$-intermediate dimensions are unchanged when a bounded set is replaced by its closure. The same holds for the Assouad dimension, the upper Assouad spectrum, the quasi-Assouad dimension, and the upper and lower box dimension profiles. In proofs concerning these quantities, we may therefore replace $E$ by $\overline E$ and work with a compact set without loss of generality.

  \section{Intermediate dimensions}\label{sec-intermediate}
Frostman-type characterizations of fractal dimensions were studied in \cite{FZ}. In this section, we first recall the mass distribution principle and Frostman's lemma for intermediate dimensions.  
  \begin{lem}[\cite{FFK} Proposition 2.2] \label{Mass}
  	Let $E$ be a non-empty and bounded  subset of $\R^d$ and let $\theta\in(0,1]$ and $s\geq0$.  Suppose that there are $c, r_0>0$ such that for all $0<r\leq r_0$ we can find a Borel measure $\mu_r\in \mP(E)$ with
  	\begin{equation*}
  		\mu_r(B(x,h))\leq c\cdot h^s \quad \text{for all   $x\in\R^d$ and  $r\leq h\leq r^\theta$}.
  	\end{equation*}
  	Then $\lid E\geq s$. Moreover, if measures $\mu_r$ with the above properties can be found only for a sequence of scales $r\to0$, then the conclusion is weakened to $\uid E\geq s$.
  \end{lem}
  \begin{lem}[\cite{FFK} Proposition 2.3; \cite{Ba} Lemma 5.2]\label{Frostman}
  	Let $E$ be a non-empty and bounded   subset of $\R^d$ and let $\theta\in(0,1]$.
  	\begin{itemize}
  		\item [$(1)$] If $0<s<\uid E$, then there exists a constant $c>0$ such that  for all $r_0>0$, there exists $0<r\leq r_0$ and a  measure $\mu_r\in \mP(E)$ such that if $x\in\R^d $ and  $r\leq h\leq r^\theta$ then
  		\begin{equation*}
  			\mu_r(B(x,h))\leq c\cdot h^s.
  		\end{equation*}
  		\item [$(2)$] If $0<s<\lid E$, then there exists a constant $c>0$ such that  for all  sufficiently small $ r>0$ there exists a  measure  $\mu_r\in \mP(E)$ such that if $x\in \R^d$ and  $r\leq h\leq r^\theta$ then
  		\begin{equation*}
  			\mu_r(B(x,h))\leq c\cdot h^s.
  		\end{equation*}
  	\end{itemize}
  \end{lem}

Combining Lemmas \ref{Mass} and \ref{Frostman}, we obtain, for $\theta\in(0,1]$,
  \begin{equation}\label{lid}
  	\lid E=\varliminf_{r\to0}\sup_{\mu\in\mP(E)}\inf_{r\leq h\leq r^\theta }\inf_x\frac{\log\mu(B(x,h))}{\log h}
  \end{equation}
  and 
  \begin{equation}\label{uid}
  	\uid E=\varlimsup_{r\to0}\sup_{\mu\in\mP(E)}\inf_{r\leq h\leq r^\theta }\inf_x\frac{\log\mu(B(x,h))}{\log h}.
  \end{equation}

Let $s\in\R$ and $0<r<1$, and let $\mu$ be a Borel probability measure supported on $E$. We define
\begin{equation*}
	\alpha(s,\mu,r)=\begin{cases}
		\sup\Big\{\alpha\in[0,1]: \sup_x\mu(B(x, r^\alpha))\geq r^{\alpha s}\Big\}, & s>0; \\
		0, & s\leq0.
	\end{cases}
\end{equation*}
 Furthermore, we define
\begin{equation}\label{alpha}
	\alpha_E(s)=\lim_{\delta\to 0^+}\varliminf_{r\to0}\inf_{\mu\in\mP(E)}\alpha(s+\delta,\mu,r)
\end{equation}
and
\begin{equation}\label{beta}
	\beta_E(s)=\lim_{\delta\to 0^+}\varlimsup_{r\to0}\inf_{\mu\in\mP(E)}\alpha(s+\delta,\mu,r).
\end{equation}
For convenience, we abbreviate $\alpha_E(s)$ and $\beta_E(s)$ as $\alpha(s)$ and $\beta(s)$, respectively. It is straightforward to check that both $\alpha(s)$ and $\beta(s)$ are non-decreasing in $s$.

The following proposition identifies $\alpha(s)$ and $\beta(s)$, defined in \eqref{alpha} and \eqref{beta}, as generalized inverses of the upper and lower intermediate dimension functions, respectively.

  \begin{pro}\label{alphaid}
Let $\alpha(s)$ and $\beta(s)$ be defined in \eqref{alpha} and \eqref{beta}, respectively. Then, for all $s>0$,
  	\begin{equation}\label{alphaid1}
  		\alpha(s)=\sup\Big\{\theta\in (0,1]: \uid(E)\leq s\Big\} 
  	\end{equation}
  	and 
  	\begin{equation}\label{alphaid2}
  		\beta(s)=\sup\Big\{\theta\in (0,1]: \lid(E)\leq s\Big\}.
  	\end{equation}
  \end{pro}
\begin{proof}
We prove \eqref{alphaid1}; the proof of \eqref{alphaid2} is similar, using all sufficiently small scales in place of a sequence of scales tending to zero, and the $\varlimsup$ in \eqref{beta} in place of the $\varliminf$ in \eqref{alpha}.
	We first prove
	\begin{equation}\label{sec3eq1}
		\alpha(s)\leq\sup\Big\{\theta\in(0,1]: \uid E\leq s\Big\}.
	\end{equation}

If the right-hand side of \eqref{sec3eq1} equals $1$, the inequality is trivial. Otherwise, fix $\theta\in(0,1]$ with $\uid E>s$, and choose $\delta_0>0$ such that $\uid E>s+2\delta_0$. By Lemma \ref{Frostman}, there exists $c>0$ such that, for arbitrarily small $r\in(0,1)$, there is a measure $\mu_r\in\mP(E)$ satisfying
\[
\sup_x\mu_r(B(x,h))\leq c h^{s+2\delta_0}
\quad\text{for all }r\leq h\leq r^\theta.
\]
Let $r$ be  small enough such that $c r^{\theta\delta_0}<1$, and hence
\[
\sup_x\mu_r(B(x,h))
\leq c h^{s+2\delta_0}
\leq c r^{\theta\delta_0}h^{s+\delta_0}
<h^{s+\delta_0}
\quad (r\leq h\leq r^\theta).
\]
It follows that $\alpha(s+\delta_0,\mu_r,r)\leq\theta$, and thus $\alpha(s)\leq\theta$. Since this holds for every $\theta\in(0,1]$ with $\uid E>s$, we conclude  \eqref{sec3eq1}.

If $\alpha(s)=1$, equality follows from \eqref{sec3eq1}. Otherwise, fix $\theta\in(\alpha(s),1]$. By \eqref{alpha}, there exists $\delta_0>0$ such that, for arbitrarily small $r\in(0,1)$, there is a measure $\mu_r\in\mP(E)$ with $\alpha(s+\delta_0,\mu_r,r)<\theta$. Since $s+\delta_0>0$, the definition of $\alpha(s+\delta_0,\mu_r,r)$ gives
	\[
	\sup_x \mu_r(B(x,h))< h^{s+\delta_0} \quad \text{for all } r\leq h\leq r^\theta.
	\]
By Lemma \ref{Mass}, we have $\uid E\geq s+\delta_0>s$. By the monotonicity of the upper intermediate dimensions, it follows that
	\[
	\theta\geq\sup\Big\{\theta\in(0,1]: \uid E\leq s\Big\}.
	\]
Letting $\theta\downarrow\alpha(s)$ yields the reverse inequality of \eqref{sec3eq1} and completes the proof.
\end{proof}
  
  \section{Box dimension profiles}\label{sec-profile}
 In this section, we establish equivalent characterizations of upper and lower box dimension profiles in Theorem \ref{main}. We first introduce several auxiliary lemmas.

  \begin{lem}\label{lem1}
Let $E\subset\R^d$ be non-empty and bounded, and let $m>0$. Then
  	\begin{equation}\label{lem11}
  		\lbd^m E=m+\varliminf_{r\rightarrow0}\sup_{\mu\in\mP(E)}\inf_x\frac{\log\int_r^\infty\mu(B(x,h))h^{-m-1}\dih}{\log r},
  	\end{equation}
  	and 
  	\begin{equation}\label{lem12}
  		\ubd^m E=m+\varlimsup_{r\rightarrow0}\sup_{\mu\in\mP(E)}\inf_x\frac{\log\int_r^\infty\mu(B(x,h))h^{-m-1}\dih}{\log r}.
  	\end{equation}
  \end{lem}
  \begin{proof}
  We prove \eqref{lem12}; \eqref{lem11} follows by using all sufficiently small scales in place of a sequence of scales tending to zero.
  We first assume that $E$ is compact.
  	By Fubini's theorem and a change of variables, it is straightforward to derive
  	\begin{equation}\label{phimu-original}
  		\phi_r^m\ast\mu(x)=mr^m\int_r^\infty\mu(B(x,h))h^{-m-1}\dih.
  	\end{equation}
  	It is clear that 
  	\begin{equation*}
  		\ubd^m E
  		\geq \varlimsup_{r\to0}\sup_{\mu\in\mP(E)}\inf_{x\in\R^d}\frac{\log\phi_r^m\ast\mu(x) }{\log r}.
  	\end{equation*}
  	We now prove the reverse inequality.
  	Fix $s<\ubd^m E$. Then for any $r_0>0$, there exist $0<r<r_0$ and a measure $\mu_r\in\mP(E)$ such that
  	\begin{equation*}
  		\int\phi_r^m\ast\mu_r(x)\dif\mu_r(x)<r^s.
  	\end{equation*}
  	It follows that
  	\begin{equation*}
  		r^s> \int\phi_r^m\ast\mu_r(x)\dif\mu_r(x)> 2r^s\cdot \mu_r\{x\in E: \phi_r^m\ast\mu_r(x)>2r^s\}.
  	\end{equation*}
  	This implies
  	\begin{equation}\label{proeq-original}
  		\mu_r\{x\in E: \phi_r^m\ast\mu_r(x)\leq2r^s\}>\frac{1}{2}.
  	\end{equation}
  	Define
  	\begin{equation}\label{proeq3-original}
  		E_r=\{x\in E: \phi_r^m\ast\mu_r(x)\leq2r^s\}\quad\text{ and } \quad
  		\widetilde{\mu}_r=\frac{\mu_r|_{E_r}}{\mu_r(E_r)}.
  	\end{equation}
  	Since $\phi_r^m(x)$ is bounded and continuous, by the dominated convergence theorem, the function $x\mapsto\phi_r^m*\mu_r(x)$ is continuous. Then the set
  	\[
  	E_r=E\cap\{x\in\mathbb R^d:\phi_r^m*\mu_r(x)\le 2r^s\}
  	\]
  	is compact.

  	If $y\in E_r$, it follows from \eqref{proeq-original} and \eqref{proeq3-original} that
  	\begin{align}\label{proeq2-original}
  		\phi_r^m\ast\widetilde{\mu}_r(y)\leq (\mu_r(E_r))^{-1} \cdot  \phi_r^m\ast\mu_r(y)
  		\leq 4 r^s.
  	\end{align}
  	Now assume $x\notin E_r$. Since $E_r$ is compact, there exists $y_x\in E_r$ such that 
  	\begin{equation*}
  		|x-y_x|=\inf\{|x-y|:y\in E_r\}.
  	\end{equation*}
  	Thus $\widetilde{\mu}_r(B(x,h))=0$ for all $0<h<|x-y_x|$.
  	Set $m_{x,r}=\max\{r, |x-y_x|\}$. Then for all $h\geq m_{x,r}$, we have
  	\begin{equation}\label{proeq4}
  		B(x,h)\subset B(y_x,2h).
  	\end{equation}
  	It follows that
  	\begin{align}\label{proeq5-original}
  		\phi_r^m\ast\widetilde{\mu}_r(x)=& mr^m\int_r^\infty\widetilde{\mu}_r(B(x,h))h^{-m-1}\dih \quad\text{(by \eqref{phimu-original})}\notag\\
  		=&mr^m\int_{m_{x,r}}^\infty\widetilde{\mu}_r(B(x,h))h^{-m-1}\dih\notag\\
  		\leq&mr^m\int_{m_{x,r}}^\infty\widetilde{\mu}_r(B(y_x,2h))h^{-m-1}\dih \quad\text{(by \eqref{proeq4})}\notag\\
  		\leq& 2^m \cdot\phi_r^m\ast\widetilde{\mu}_r(y_x)\notag\\
  		\leq& 2^{m+2}  r^s \quad\text{(by \eqref{proeq2-original}).}
  	\end{align}
  	We conclude from \eqref{proeq2-original} and \eqref{proeq5-original} that for all $x\in\R^d$,
  	\begin{equation*}
  		\phi_r^m\ast\widetilde{\mu}_r(x)\leq 2^{m+2}  r^s.
  	\end{equation*}
  	Therefore, taking the infimum over $x\in\R^d$ and the supremum over measures $\mu\in\mathcal{P}(E)$ yields
  	\begin{equation*}
  		\varlimsup_{r\to0}\sup_{\mu\in\mP(E)}\inf_{x\in\R^d}\frac{\log\phi_r^m\ast\mu(x) }{\log r}\geq s.
  	\end{equation*}
  	Letting $s\to\ubd^m E$, we obtain
  	\begin{equation*}
  		\ubd^m E\leq\varlimsup_{r\to0}\sup_{\mu\in\mP(E)}\inf_{x\in\R^d}\frac{\log\phi_r^m\ast\mu(x) }{\log r}.
  	\end{equation*}
  	Thus we have proved \eqref{lem12} when $E$ is compact.
  	
  	It is clear that $\overline E$ is compact for all non-empty bounded set $E$, and $E$ is dense in $\overline E$.
  	Fix $0<r<1$ and $\mu\in\mathcal{P}(\overline{E})$.
  	For any $0<\delta<1$, there exists  a finite Borel partition $\overline E=\bigcup_{i=1}^N A_i$ and points
  	$y_i\in E$ such that $|y-y_i|<\delta$ for every $y\in A_i$.
  	 Define
  	\[
  	\mu_\delta=\sum_{i=1}^N\mu(A_i)\delta_{y_i}\in\mP(E).
  	\]
  	Since $\phi_r^m$ is Lipschitz with constant $\frac{m}{r}$, we have
  	\[
  	\sup_{x\in\R^d}
  	\left|(\phi_r^m*\mu_\delta)(x)-(\phi_r^m*\mu)(x)\right|
  	\leq \frac{m\delta}{r}\longrightarrow0
  	\qquad\text{ as }\delta\to0.
  	\] 
  Thus
  	\[
  	 \inf_{x}
  	\frac{\log(\phi_r^m*\mu_\delta(x))}{\log r}\to \inf_{x}
  	\frac{\log(\phi_r^m*\mu(x))}{\log r}\quad\text{as }\delta\to0.
  	\]
  	Together with $\mP(E)\subset\mP(\overline E)$, this yields
  	\[
  	\sup_{\nu\in\mP(E)} \inf_{x}
  	\frac{\log(\phi_r^m*\nu(x))}{\log r}
  	=
  	\sup_{\nu\in\mP(\overline E)} \inf_{x}
  	\frac{\log(\phi_r^m*\nu(x))}{\log r}.
  	\]
   Then  \eqref{lem12} also holds for every non-empty bounded set $E$.
   This completes the proof.
  \end{proof}

Fix $m>0$. Let $0<s<m$, $0<r<1$, and $\mu\in\mP(E)$. We set
$$s_0(s,\mu,r)=s \quad \text{ and } \quad \alpha_0(s,\mu,r)=\alpha(s,\mu,r),$$ 
	and for each $k\geq1$, we inductively define
	\begin{equation}\label{defisk}
		s_k(s,\mu,r)=
		\begin{cases}
			0, & \alpha_{k-1}(s,\mu,r)=0,\\[1mm]
			\displaystyle\max\Big\{0,\ m-\frac{m-s}{\alpha_{k-1}(s,\mu,r)}\Big\},
			& \alpha_{k-1}(s,\mu,r)>0,
		\end{cases}
	\end{equation}
	and
	\begin{equation}\label{defialphak}
		\alpha_k(s,\mu,r)=\alpha(s_k(s,\mu,r),\mu,r). 
	\end{equation}
 
  In the following lemma, we discuss some properties of 	 $\ak$  and $\sk$  defined in \eqref{defisk} and \eqref{defialphak}.

  \begin{lem}\label{alphak}
Let $0<s<m$, $0<r<1$, and $\mu\in\mP(E)$.
  	The following statements hold:
  	\begin{itemize}
  		\item [$(1)$] $\ak$ and $\sk$ are non-increasing in $k$.
  		\item [$(2)$] If $\ak>0$, then 
  		$$\sup_x\mu(B(x,r^{\ak}))\geq r^{\ak\sk}.$$
  	\end{itemize} 
  \end{lem}
  \begin{proof}
	  	$(1)$ As $\alpha(s,\mu,r)$ is non-decreasing in $s$, it suffices to prove that $s_k(s,\mu,r)$ is non-increasing in $k$.
	  	Since $\alpha(s,\mu,r)\in[0,1]$, it is clear that $s_0(s,\mu,r)\geq s_1(s,\mu,r)$. 
	  	Assume that $s_{k-1}(s,\mu,r)\geq s_{k}(s,\mu,r)$ for some $k\geq1$. It follows that $\alpha_{k-1}(s,\mu,r)\geq\alpha_{k}(s,\mu,r)$.
  	Then,
  	\begin{equation*}
	  		s_{k}(s,\mu,r)\geq s_{k+1}(s,\mu,r).
	  	\end{equation*}
	  	Thus, $s_k(s,\mu,r)$ and $\alpha_{k}(s,\mu,r)$ are non-increasing in $k$.

$(2)$ Assume that $\alpha_k(s,\mu,r)>0$. Then $s_k(s,\mu,r)>0$.
By the definition of $\alpha_k(s,\mu,r)$, there are exponents $a_j\in[0,1]$ converging to $\alpha_k(s,\mu,r)$ such that
\[
\sup_x\mu(B(x,r^{a_j}))\geq r^{a_js_k(s,\mu,r)}.
\]
Choose $x_j\in\R^d$ satisfying
\[
\mu(B(x_j,r^{a_j}))\geq r^{a_js_k(s,\mu,r)}-\frac1j.
\]
Let \(x_0\in\mathbb R^d\) be an accumulation point of the bounded sequence $\{x_j\}_j$; by passing to a subsequence and relabelling, we may assume without loss of generality that  $x_j\to x_0\in\R^d$.
For each $\delta>0$, we have
\[
B(x_j,r^{a_j})\subset B(x_0,r^{\alpha_k(s,\mu,r)}+\delta)
\]
for all sufficiently large $j$. Taking limits gives
\[
r^{\alpha_k(s,\mu,r)s_k(s,\mu,r)}
\leq\mu(B(x_0,r^{\alpha_k(s,\mu,r)}+\delta)).
\]
Letting $\delta\downarrow0$, we obtain
\[
r^{\alpha_k(s,\mu,r)s_k(s,\mu,r)}
\leq\mu(B(x_0,r^{\alpha_k(s,\mu,r)}))
\leq\sup_x\mu(B(x,r^{\alpha_k(s,\mu,r)})).
\]
  \end{proof}

We now present several equivalent characterizations of the upper box dimension profile using the notation  $s_k(s,\mu,r)$ and  $\alpha_k(s,\mu,r)$ defined in \eqref{defisk} and \eqref{defialphak}.

 \begin{thmm}\label{main}
 	Let $E$ be a non-empty and bounded set. Then
 	\begin{align}
 		\ubd^m E
 		=&\sup\Big\{0<s<m: \lim_{k\to\infty}\varliminf_{r\to0}\inf_{\mu\in\mP(E)}\alpha_k(s,\mu,r)=0\Big\}\label{eq1}\\
 		=&\inf\Big\{0<s<m: \lim_{k\to\infty}\varliminf_{r\to0}\inf_{\mu\in\mP(E)}\alpha_k(s,\mu,r)>0\Big\}\label{eq2}\\
	 	=&\inf\Big\{0<s<m:  \lim_{k\to\infty}\varlimsup_{r\to0}\sup_{\mu\in\mP(E)}\inf_{r\leq h\leq r^{\frac{m-s}{m}}}\inf_x\frac{\log\mu(B(x,h))}{s_k(s,\mu,r)\log h}\leq 1 \Big\}\label{eq3}\\
 		=&\inf\Big\{0<s<m:  \varlimsup_{r\to0}\sup_{\mu\in\mP(E)}\inf_{\frac{m-s}{m}<\alpha\leq1}\inf_x\frac{\log\mu(B(x,r^\alpha))}{(s-m(1-\alpha))\log r}\leq 1 \Big\},\label{eq4}
 	\end{align}
 	where the infimum of the empty subset of $(0,m)$ is taken to be $m$.
 	In \eqref{eq3}, whenever $s_k(s,\mu,r)=0$, the quotient
 	$
 	\frac{\log\mu(B(x,h))}{s_k(s,\mu,r)\log h}
 	$
 	is defined to be $+\infty$ for all $x\in\R^d$ and
 	$r\leq h\leq r^{\frac{m-s}{m}}$.
 	Analogous statements hold for the lower box dimension profile $\lbd^m E$ by replacing $\varliminf_{r\to0}$ with $\varlimsup_{r\to0}$ in \eqref{eq1} and \eqref{eq2}, and replacing $\varlimsup_{r\to0}$ with $\varliminf_{r\to0}$ in \eqref{eq3} and \eqref{eq4}.
 \end{thmm}
  \begin{proof}
(1) We first prove \eqref{eq1} and \eqref{eq2}. Let $s\in(0,m)$.
  	If
  	$$\lim_{k\to\infty}\varliminf_{r\to0}\inf_{\mu\in\mP(E)}\alpha_k(s,\mu,r)=0,$$
  	then there exists an integer $N\in\mathbb{N}$ such that  
  	\begin{equation*}
  		\varliminf_{r\to0}\inf_{\mu\in\mP(E)}\alpha_N(s,\mu,r)<\frac{m-s}{m}.
  	\end{equation*}
  	Thus, for any $r_0>0$, there exist $0<r<r_0$ and a measure $\mu_r\in\mP(E)$ such that 
  	\begin{equation}\label{theq1}
  		\alpha_N(s,\mu_r,r)<\frac{m-s}{m}.
  	\end{equation}
	  	For $0\leq k\leq N-1$, if $\alpha_k(s,\mu_r,r)\leq\frac{m-s}{m}$, it follows from \eqref{defisk} and \eqref{defialphak} that $$\alpha_{k+1}(s,\mu_r,r)=s_{k+1}(s,\mu_r,r)=0.$$
  	Then
  	\begin{equation}\label{theq5}
  		\sup_x\int_{r^{\alpha_k(s,\mu_r,r)}}^{r^{\alpha_{k+1}(s,\mu_r,r)}}\mu_r(B(x,h))h^{-m-1}\dih \leq \int_{r^{\alpha_k(s,\mu_r,r)}}^{\infty} h^{-m-1}\dih  \leq \frac{1}{m}r^{s-m}.
  	\end{equation}
	  	If $\alpha_k(s,\mu_r,r)>\frac{m-s}{m}$, the definition of $\alpha_{k+1}(s,\mu_r,r)$   ensures 
\[
\sup_x\mu_r(B(x,h))\leq h^{s_{k+1}(s,\mu_r,r)}
\quad\text{for }r\leq h<r^{\alpha_{k+1}(s,\mu_r,r)}.
\]
Consequently,
  	\begin{align}\label{theq14}
  		\sup_x\int_{r^{\alpha_k(s,\mu_r,r)}}^{r^{\alpha_{k+1}(s,\mu_r,r)}}\mu_r(B(x,h))h^{-m-1}\dih
	  		&\leq\int_{r^{\alpha_k(s,\mu_r,r)}}^{r^{\alpha_{k+1}(s,\mu_r,r)}} h^{s_{k+1}(s,\mu_r,r)-m-1}\dih\notag\\
	  		&\leq\frac{1}{m-s_{k+1}(s,\mu_r,r)} r^{\alpha_k(s,\mu_r,r)(s_{k+1}(s,\mu_r,r)-m)}\notag\\
  		&\leq\frac{1}{m-s} r^{s-m}.
  	\end{align}
  	It follows that for all $x\in\R^d$,
  	\begin{align*}
  		&\int_r^{\infty}\mu_r(B(x,h))h^{-m-1}\dih\\
  		&\leq\int_r^{r^{\alpha(s,\mu_r,r)}}\mu_r(B(x,h))h^{-m-1}\dih+\sum_{k=0}^{N-1}\int_{r^{\alpha_k(s,\mu_r,r)}}^{r^{\alpha_{k+1}(s,\mu_r,r)}}\mu_r(B(x,h))h^{-m-1}\dih+\int_{r^{\alpha_N(s,\mu_r,r)}}^\infty h^{-m-1}\dih\\
  		&\leq\int_r^{r^{\alpha(s,\mu_r,r)}}h^{s-m-1}\dih+\frac{N}{m-s} r^{s-m}+\frac{1}{m} r^{-m\alpha_{N}(s,\mu_r,r)}\quad\text{(by \eqref{theq5} and \eqref{theq14})}\\
  		&\leq \frac{1}{m-s} r^{s-m} + \frac{N}{m-s} r^{s-m} + \frac{1}{m} r^{s-m} \quad\text{(by \eqref{theq1})}\\
  		&\leq\frac{N+2}{m-s} r^{s-m}.
  	\end{align*}
  	According to Lemma \ref{lem1}, we have $\ubd^m E\geq s$. 
  	Consequently, 
  	\begin{equation*}
  		\ubd^m E\geq\sup\Big\{0<s<m: \lim_{k\to\infty}\varliminf_{r\to0}\inf_{\mu\in\mP(E)}\alpha_k(s,\mu,r)=0\Big\}.
  	\end{equation*}
  	
  	We now prove the reverse inequality.
  	Assume that $0<s<m$ is such that
  	$$L(s):=\lim_{k\to\infty}\varliminf_{r\to0}\inf_{\mu\in\mP(E)}\alpha_k(s,\mu,r)>0.$$ 
  	Since $\alpha_j(s,\mu,r)$ is non-increasing in $j$, for each fixed
  	$k\geq1$ we have
  	\[
  	\varliminf_{r\to0}\inf_{\mu\in\mP(E)}
  	\alpha_{k+1}(s,\mu,r)\geq L(s)>0.
  	\]
  	Consequently, there exists $\rho_k>0$ such that
  	$\alpha_{k+1}(s,\mu,r)>0$ for all $0<r<\rho_k$ and
  	$\mu\in\mP(E)$.
  	By the definition of $\alpha$ and \eqref{defialphak}, this implies
  	$s_{k+1}(s,\mu,r)>0$, and hence \eqref{defisk} gives
  	\[
  	\alpha_k(s,\mu,r)>\frac{m-s}{m}.
  	\]
  	Choose  a decreasing sequence $\{R_k\}_{k=1}^\infty \downarrow 0$ such that for all $0<r\leq R_k$ and $\mu\in\mP(E)$,  
  	\begin{equation}\label{theq2-original}
  		\alpha_k(s,\mu,r)>\frac{m-s}{m}.
  	\end{equation}
  	Fix $k\in\mathbb{N}$, $0<r\leq R_k$, and $\mu\in\mP(E)$. 
  	For all $1\leq j\leq k$, since $0<r\leq R_k\leq R_j$, we have
  	$$\alpha_j(s,\mu,r)>\frac{m-s}{m}.$$ 
  	Thus,
\begin{equation*}
	  		s_j(s,\mu,r)\geq s_{j+1}(s,\mu,r)= m-\frac{m-s}{\alpha_j(s,\mu,r)} >0.
  	\end{equation*}
  	It follows that  
  	\begin{align*}
  		\sup_x\int_r^\infty\mu(B(x,h))h^{-m-1}\dih
  		&\geq\sup_x\int_{r^{\alpha_j(s,\mu,r)}}^\infty\mu(B(x,h))h^{-m-1}\dih\\
  		&\geq\sup_x \mu(B(x,{r^{\alpha_j(s,\mu,r)}}))\int_{r^{\alpha_j(s,\mu,r)}}^\infty h^{-m-1}\dih\\
	  		&\geq r^{\alpha_j(s,\mu,r)s_j(s,\mu,r)}\int_{r^{\alpha_j(s,\mu,r)}}^\infty h^{-m-1}\dih\quad\text{(by Lemma \ref{alphak})}\\
	  		&=\frac{1}{m} r^{\alpha_j(s,\mu,r)(s_j(s,\mu,r)-m)}\\
	  		&=\frac{1}{m} r^{s-m+\alpha_j(s,\mu,r)(s_{j}(s,\mu,r)-s_{j+1}(s,\mu,r) )}. 
  	\end{align*}
Consequently,
  	\begin{equation*}
	  		\inf_x\frac{\log\int_r^\infty\mu(B(x,h))h^{-m-1}\dih}{\log r}\leq s-m+\inf_{1\leq j\leq k}\alpha_j(s,\mu,r)(s_{j}(s,\mu,r)-s_{j+1}(s,\mu,r) )-\frac{\log m}{\log r}.
  	\end{equation*}
  	According to Lemma \ref{lem1} and the arbitrariness of $k$, we have 
  	\begin{equation}\label{theq6}
	  		\ubd^m E\leq s + \lim_{k\to\infty}\varlimsup_{r\to0}\sup_{\mu\in\mP(E)}\inf_{1\leq j\leq k}\alpha_j(s,\mu,r)(s_{j}(s,\mu,r)-s_{j+1}(s,\mu,r)).
  	\end{equation}
  	We claim that  
  	\begin{equation}\label{theq7}
	  		\lim_{k\to\infty} \varlimsup_{r\to0}\sup_{\mu\in\mP(E)}\inf_{1\leq j\leq k}\alpha_j(s,\mu,r)(s_{j}(s,\mu,r)-s_{j+1}(s,\mu,r))=0.
  	\end{equation}
  	We prove \eqref{theq7} by contradiction. Assume there exists an $0<\epsilon<m-s$ such that
  	\begin{equation*}
	  		\lim_{k\to\infty}\varlimsup_{r\to0}\sup_{\mu\in\mP(E)}\inf_{1\leq j\leq k}\alpha_j(s,\mu,r)(s_{j}(s,\mu,r)-s_{j+1}(s,\mu,r))>\epsilon. 
  	\end{equation*}
  	Let $k_0$ be the least integer such that 
  	\begin{equation*}
  		\left(1-\frac{\epsilon}{m-s}\right)^{k_0}<\frac{m-s}{m}.
  	\end{equation*}
  	Then there exist $0<r_0<R_{k_0}$ and $\mu_{0}\in\mP(E)$ such that 
  	\begin{equation*}
	  		\inf_{1\leq j\leq k_0}\alpha_j(s,\mu_0,r_0)(s_{j}(s,\mu_0,r_0)-s_{j+1}(s,\mu_0,r_0))>\epsilon.
  	\end{equation*}
	  	It is clear that for all $1\leq j\leq k_0$, $\alpha_j(s,\mu_0,r_0)>0$ and $s_{j}(s,\mu_0,r_0)>0$.
  	According to \eqref{defisk}, we have
  	\begin{equation*}
	  		\epsilon<\alpha_j(s,\mu_0,r_0)(s_{j}(s,\mu_0,r_0)-s_{j+1}(s,\mu_0,r_0))\leq (m-s)\left(1 - \frac{\alpha_j(s,\mu_0,r_0)}{\alpha_{j-1}(s,\mu_0,r_0)}\right).
  	\end{equation*}
  	Thus, for all $1\leq j\leq k_0$,
  	\begin{equation*}
  		\frac{\alpha_j(s,\mu_0,r_0)}{\alpha_{j-1}(s,\mu_0,r_0)}<1-\frac{\epsilon}{m-s}.
  	\end{equation*}
  	Telescoping the product, we deduce
  	\begin{equation*}
  		\alpha_{k_0}(s,\mu_0,r_0)=\prod_{j=1}^{k_0}\frac{\alpha_j(s,\mu_0,r_0)}{\alpha_{j-1}(s,\mu_0,r_0)}\cdot \alpha_0(s,\mu_0,r_0)<\left(1-\frac{\epsilon}{m-s}\right)^{k_0}<\frac{m-s}{m},
  	\end{equation*}
  	which directly contradicts \eqref{theq2-original}. 
  	Therefore, claim \eqref{theq7} holds. 
  	Substituting \eqref{theq7} into \eqref{theq6} gives $\ubd^m E\leq s$. 
  	Thus we conclude that
  	\begin{equation*}
  		\ubd^m E\leq\sup\Big\{0<s<m: \lim_{k\to\infty}\varliminf_{r\to0}\inf_{\mu\in\mP(E)}\alpha_k(s,\mu,r)=0\Big\}.
  	\end{equation*}
  	Since $\lim_{k\to\infty}\varliminf_{r\to0}\inf_{\mu\in\mP(E)}\alpha_k(s,\mu,r)$ is non-decreasing in $s$, it is clear that \eqref{eq1} and \eqref{eq2} are equivalent.
  	This has completed the proof of \eqref{eq1} and \eqref{eq2}.

  	(2) Let $s\in(0,m)$.  If
  	\[
  	L(s):=\lim_{k\to\infty}\varliminf_{r\to0}
  	\inf_{\mu\in\mP(E)}\alpha_k(s,\mu,r)>0,
  	\]
   the argument in part (1) shows that, for each fixed
  	$k\geq1$, all sufficiently small $r>0$ and all $\mu\in\mP(E)$
  	satisfy $\alpha_k(s,\mu,r)>\frac{m-s}{m}$ and $s_k(s,\mu,r)>0$.
  	 Consequently,
  	\[
  	\lim_{k\to\infty}\varlimsup_{r\to0}
  	\sup_{\mu\in\mP(E)}	\inf_{r\leq h\leq r^{\frac{m-s}{m}}}\inf_x\frac{\log\mu(B(x,h))}{s_k(s,\mu,r)\log h}\leq1.
  	\]
  	If $L(s)=0$, then for all $k$ large enough,
  	\[
  	\varliminf_{r\to0}\inf_{\mu\in\mP(E)}
  	\alpha_k(s,\mu,r)<\frac{m-s}{m}.
  	\]
  	Thus there exists arbitrarily small $r>0$ and  $\mu_r\in\mP(E)$
  	with $\alpha_k(s,\mu_r,r)<\frac{m-s}{m}$.
  	Clearly, $s_{k+1}(s,\mu_r,r)=0$. 
  	  Hence,  
  	\[
  	\varlimsup_{r\to0}\sup_{\mu\in\mP(E)}\inf_{r\leq h\leq r^{\frac{m-s}{m}}}\inf_x\frac{\log\mu(B(x,h))}{s_{k+1 }(s,\mu,r)\log h}=+\infty.
  	\]
  	These two cases show that
  	\[
  	L(s)>0
  	\quad\Longleftrightarrow\quad
  	\lim_{k\to\infty}\varlimsup_{r\to0}
  	\sup_{\mu\in\mP(E)}	\inf_{r\leq h\leq r^{\frac{m-s}{m}}}\inf_x\frac{\log\mu(B(x,h))}{s_k(s,\mu,r)\log h}\leq1,
  	\]
  	which proves the equivalence of \eqref{eq2} and \eqref{eq3}.

(3) It remains to prove \eqref{eq4}. Fix $s\in(0,m)$ strictly below its right-hand side.
Then there exists $\epsilon>0$ such that, for arbitrarily small $r\in(0,1)$, there is a measure $\mu_r\in\mP(E)$ satisfying
\[
\inf_{\frac{m-s}{m}<\alpha\leq1}\inf_x
\frac{\log\mu_r(B(x,r^\alpha))}
{(s-m(1-\alpha))\log r}>1+\epsilon.
\]
Consequently, for every $x\in\R^d$ and $r\leq h<r^{\frac{m-s}{m}}$,
\[
\mu_r(B(x,h))\leq r^{(s-m)(1+\epsilon)}h^{m(1+\epsilon)}.
\]
It follows that
\begin{align*}
\int_r^\infty\mu_r(B(x,h))h^{-m-1}\dih
\leq r^{(s-m)(1+\epsilon)}
\int_r^{r^{\frac{m-s}{m}}}h^{m\epsilon-1}\dih +\int_{r^{\frac{m-s}{m}}}^\infty h^{-m-1}\dih 
\leq\frac{1+\epsilon^{-1}}m\,r^{s-m}.
\end{align*}
By Lemma \ref{lem1}, $\ubd^m E\geq s$.
Then we obtain
\begin{equation*}
	\ubd^m E\geq\inf\Big\{0<s<m:  \varlimsup_{r\to0}\sup_{\mu\in\mP(E)}\inf_{\frac{m-s}{m}<\alpha\leq1}\inf_x\frac{\log\mu(B(x,r^\alpha))}{(s-m(1-\alpha))\log r}\leq 1 \Big\}.
\end{equation*}

Conversely, suppose that $s\in(0,m)$ satisfies
\[
\varlimsup_{r\to0}\sup_{\mu\in\mP(E)}
\inf_{\frac{m-s}{m}<\alpha\leq1}\inf_x
\frac{\log\mu(B(x,r^\alpha))}
{(s-m(1-\alpha))\log r}\leq1.
\]
Fix $\delta>0$. For all sufficiently small $r\in(0,1)$ and every $\mu\in\mP(E)$, there exist $\alpha\in(\frac{m-s}{m},1]$ and $x\in\R^d$ such that
\[
\mu(B(x,r^\alpha))>r^{(1+\delta)(s-m+m\alpha)}.
\]
It follows that
\begin{align*}
\sup_y\int_r^\infty\mu(B(y,h))h^{-m-1}\dih&\geq\int_{r^\alpha}^\infty\mu(B(x,h))h^{-m-1}\dih\\
&\geq\frac1m r^{-m\alpha}\mu(B(x,r^\alpha))\\
&>\frac1m r^{s-m+\delta(s-m+m\alpha)}\\
&\geq\frac1m r^{s(1+\delta)-m}.
\end{align*}
According to Lemma \ref{lem1},  we have $\ubd^m E\leq s(1+\delta)$.
Letting $\delta\downarrow0$ yields $\ubd^m E\leq s$.
This proves the reverse inequality in \eqref{eq4}.

The proof for the lower box dimension profile is analogous, with the scale quantifiers ``for all sufficiently small \(r>0\)'' and ``for arbitrarily small \(r>0\)'' interchanged.
  \end{proof}
  
  \section{Estimating box dimension profiles using intermediate dimensions}\label{sec:proof}

  In this section, we will explore the connection between box dimension profiles and intermediate dimensions.
 
  \begin{pro}\label{thm:upper}
 	Let $E$ be a non-empty and bounded set. Then
 	\begin{equation}\label{upper1}
 		\ubd^mE
 		\leq
 		\inf_{\theta\in(0,1]}
 		\left\{
 		m(1-\theta)+\theta\cdot\uid E
 		\right\}
 	\end{equation}
 	and
 	\begin{equation}\label{upper2}
 		\lbd^mE
 		\leq
 		\inf_{\theta\in(0,1]}
 		\left\{
 		m(1-\theta)+\theta\cdot\lid E
 		\right\}.
 	\end{equation}
 \end{pro}
 
 \begin{proof}
 		We only prove \eqref{upper1}. Since the proof of \eqref{upper2} is very similar, we omit it. Fix $\theta\in(0,1]$.
 		If $\uid E\geq m$, then
 		\[
 		\ubd^m E\leq m\leq m(1-\theta)+\theta\cdot\uid E.
 		\]
 		Otherwise, choose $\uid E<t<m$ and set
 		\[
 		s=m(1-\theta)+\theta t.
 		\]
 		It is clear that $0<s<m$ and $s\geq t$.
 		According to \eqref{uid}, there exists $r_0\in(0,1)$ such that for all $0<r<r_0$,
 		\[
 		\sup_{\mu\in\mP(E)}\inf_{r\leq h\leq r^\theta}\inf_x
 		\frac{\log\mu(B(x,h))}{\log h}<t.
 		\]
 		Thus, for each $0<r<r_0$ and each $\mu\in\mP(E)$, there exist
 		$h\in[r,r^\theta]$ and $x\in\R^d$ such that
 		\[
 		\mu(B(x,h))>h^t.
 		\]
 		Thus
 		\[
 		\alpha(t,\mu,r)\geq\theta.
 		\]
 		
 		We next prove by mathematical induction that for all $k\geq0$, all $0<r<r_0$,
 		and all $\mu\in\mP(E)$,
 		\[
 		s_k(s,\mu,r)\geq t
 		\quad\text{and}\quad
 		\alpha_k(s,\mu,r)\geq\theta.
 		\]
 		Since $s_0(s,\mu,r)=s\geq t$ and $\alpha(s,\mu,r)$ is non-decreasing in $s$,
 		the statement holds for $k=0$.
 		Assume that it holds for some $k\geq0$.
 		It follows from \eqref{defisk} that
 		\begin{align*}
 			s_{k+1}(s,\mu,r)
 			=\max\left\{0,m-\frac{m-s}{\alpha_k(s,\mu,r)}\right\}
 			\geq m-\frac{m-s}{\theta}=t.
 		\end{align*}
 		Consequently,
 		\[
 		\alpha_{k+1}(s,\mu,r)
 		=\alpha(s_{k+1}(s,\mu,r),\mu,r)
 		\geq\alpha(t,\mu,r)\geq\theta.
 		\]
 		This completes the induction. Therefore,
 		\[
 		\lim_{k\to\infty}\varliminf_{r\to0}
 		\inf_{\mu\in\mP(E)}\alpha_k(s,\mu,r)\geq\theta>0.
 		\]
 		By \eqref{eq2}, we have
 		\[
 		\ubd^m E\leq s=m(1-\theta)+\theta t.
 		\]
 		Letting $t\downarrow\uid E$ and taking the infimum over $\theta\in(0,1]$
 		proves \eqref{upper1}.
\end{proof}

By Marstrand's projection theorem, if a set has equal Hausdorff and box
dimensions, then 
\begin{equation*}
	\dim_B P_V(E)=\min\{m,\dim_B E\}
\end{equation*}
	for $\gamma_{d,m}$-almost all $V\in G(d,m)$.  However, Proposition \ref{thm:upper} shows that
when a set exhibits highly irregular behavior across different   scales, in
other words, when
\[
\inf_{\theta\in(0,1]}
\left\{
m(1-\theta)+\theta\cdot\uid E
\right\}
\]
is attained at neither endpoint,   it is clear
\begin{equation}\label{drop}
	\ubd^m E<\min\{m,\ubd E\}.
\end{equation}
 In the following example, we show that when taking the  product of a highly regular set and a less regular set, a dimension drop in \eqref{drop} may occur.
 
 \begin{ex}\label{ex}
 	Let $F_p = \{0\} \cup \{n^{-p}\}_{n=1}^\infty$ for some $p > 0$, and let $A\subset\R$ be a set with
 	\begin{equation*}
 		\dim_H A = \dim_B A = c \in\big((\frac{p}{p+1})^2,1\big).
 	\end{equation*}
 	It was proved in \cite[Proposition 3.1]{FFK} that $$\dim_\theta F_p = \frac{\theta}{p+\theta}.$$
 Since
 	\begin{equation*}
 		\dim_\theta A + \dim_\theta F_p \leq \dim_\theta (A \times F_p) \leq\dim_B A+ \dim_\theta F_p
 	\end{equation*}
 	(see \cite[Proposition 2.5]{FFK}), we have
 	\begin{equation}\label{ex1eq1}
 		\dim_\theta (A \times F_p) = c + \frac{\theta}{p+\theta}.
 	\end{equation}
 	Let $$g(\theta)=1 - \theta + \theta\cdot \dim_{\theta} (A \times F_p)=1 - \theta + \theta \left( c + \frac{\theta}{p+\theta} \right).$$ 
 	Moreover, it is easy to get
 	\begin{equation*}
 		s_0 = \inf_{\theta\in(0,1]}g(\theta)=g\left( p (\frac{1}{\sqrt{c}} - 1)\right) =1 - (1-\sqrt{c})^2 p.
 	\end{equation*}
 	Applying Proposition \ref{thm:upper}, we conclude that a strict dimension drop occurs, that is,
 	\begin{equation*}
 		\ubd^1 (A \times F_p) \leq   1 - (1-\sqrt{c})^2 p < \min\left\{1, c + \frac{1}{p+1}\right\}.
 	\end{equation*}
 \end{ex}

 {
The following theorem improves the upper bound in Proposition~\ref{thm:upper}, although Proposition~\ref{thm:upper} has a simpler form.

 \begin{thmm}\label{upperbound2}
    Let $E\subset\R^d$ be non-empty and bounded.
     Then
    \begin{equation}\label{upperbound2-upper}
        \ubd^mE
        \leq
        \inf_{\theta\in(0,1]}
        \left\{\frac{m\uid E}
        {m\theta+(1-\theta)\uid E}
        \right\}
    \end{equation}
    and
    \begin{equation}\label{upperbound2-lower}
        \lbd^mE
        \leq
        \inf_{\theta\in(0,1]}
        \left\{\frac{m\lid E}
        {m\theta+(1-\theta)\lid E}
        \right\}.
    \end{equation}
 \end{thmm}

 \begin{proof}
  	Without loss of generality, we assume that there exists $\theta\in(0,1]$ such that $0<\uid E<m$.
     Fix $\theta$ and suppose that $\uid E<s<m$. Then there exists $r_0>0$ such that for any   $0<r<r_0$ there exists  a finite cover $\{U_i\}_{i=1}^{M_r}$of $E$,   
    with $U_i$ closed and 
  
\begin{equation*}
        r\leq |U_i|\leq r^\theta,
        \qquad \sum_{i} |U_i|^s\leq1.
    \end{equation*}
    Thus there is a   partition $\{A_i\}_i$ of $E$
    with $A_i\subset U_i$.
    Let
   \[
   \rho=r^a,\text{ where } a=\frac{(1-\theta)s+m\theta}{m}.
   \] 
    If $x,y\in A_i$, then $\phi_\rho^m(x-y)\geq \phi_\rho^m(|U_i|)$. Hence, for every
    $\mu\in\mP(E)$, we have
    \begin{align*}
        \iint\phi_\rho^m(x-y)\dif\mu(y)\dif\mu(x)
        \geq& \sum_i \phi_\rho^m(|U_i|)\mu(A_i)^2\\
        \geq& \frac{\big(\sum_i\mu(A_i)\big)^2}
        {\sum_i\phi_\rho^m(|U_i|)^{-1}}\quad\text{(using Cauchy--Schwarz inequality)}\\
        =&\left(\sum_i
        \max\{1,\left(\frac{|U_i|}{\rho}\right)^m\}\right)^{-1}.
    \end{align*}
    Taking the infimum over $\mu$  yields
    \begin{align*}
        C_\rho^m(E)
        &\leq \sum_{i=1}^{M_r}\max\{1,\left(\frac{|U_i|}{\rho}\right)^m\}\notag\\
        &\leq {M_r}+\rho^{-m}\sum_{i=1}^{M_r}|U_i|^m\notag\\
        &\leq \big(r^{-s}+\rho^{-m}r^{\theta(m-s)}\big)
        \sum_{i=1}^{M_r}|U_i|^s\notag\\
        &\leq r^{-s}+\rho^{-m}r^{\theta(m-s)}\\
        &\leq 2r^{-s}.  
    \end{align*}
  Clearly,
\begin{equation*}
        \frac{\log C_{r^a}^m(E)}{-\log r^a}
        \leq \frac{s}{a}+\frac{\log2}{a(-\log r)}.
    \end{equation*}

 It follows that
    \[
        \ubd^m E\leq\frac{s}{a}
        =\frac{ms}{m\theta+(1-\theta)s}.
    \]
    Letting $s\searrow \uid E$ gives  
    \[
        \ubd^m E 
        \leq\frac{m\uid E}{m\theta+(1-\theta)\uid E}.
    \]
     Taking the infimum over $\theta\in(0,1]$ proves
    \eqref{upperbound2-upper}.

    The proof of \eqref{upperbound2-lower} is analogous to that of
    \eqref{upperbound2-upper}: one only replaces the assertion that the
    cover exists for every sufficiently small $r$ by the existence of a
    sequence $\{r_i\}_i$ with $r_i\downarrow0$. Thus we omit the proof. 
 \end{proof}
 }

\begin{remark}\label{remark1}
 	Theorem \ref{upperbound2} improves the upper bound in
 	Proposition \ref{thm:upper}. Indeed, for every $\theta\in(0,1]$,
 	\begin{align*}
 		m(1-\theta)+\theta\cdot\uid E
 		-\frac{m\uid E}{m\theta+(1-\theta)\uid E}=
 		\frac{\theta(1-\theta)(m-\uid E)^2}
 		{m\theta+(1-\theta)\uid E}
 		\geq0.
 	\end{align*}
 	In particular, applying Theorem \ref{upperbound2} to the set 
 	$A\times F_p$ from Example \ref{ex} yields
 	\[
 	\ubd^1(A\times F_p)
 	\leq\frac{(1+c)^2}{(1+c)^2+p(1-c)^2}
 	<1-(1-\sqrt{c})^2p.
 	\]
    \end{remark}

We now turn to lower bounds for box dimension profiles. To formulate
these bounds, we first introduce a quantity determined by the upper
Assouad spectrum.

\begin{defi}\label{def:Km}
    Let $E\subset\R^d$ be non-empty and bounded, and let $m>0$.
    For $s\geq0$, define
 
\begin{equation*}
        K_m(s)=\max\left\{0,\sup_{0<\eta<1}
        \left[m\eta-
        \left(s-(1-\eta)\overline{\dim}_A^\eta E\right)_+\right]
        \right\},
    \end{equation*}
    where $u_+=\max\{u,0\}$.
\end{defi}
If $\dim_{qA}E=0$, then  the definition of $K_m$ gives $K_m(0)=m$. If
$\dim_{qA}E>0$, let 
    $0< s\leq\dim_{qA} E$, and let
\begin{align*}
	f(\eta):&=m\eta-\big(s-(1-\eta)\dim_{qA} E\big)_+\\
	& =\begin{cases}
		m\eta, \quad&0<\eta\leq1-\frac{s}{\dim_{qA}E};\\
		(m-\dim_{qA}E)\eta+\dim_{qA}E-s,\quad&1-\frac{s}{\dim_{qA}E}\leq\eta<1.
	\end{cases}
\end{align*}
Thus,
\begin{align}\label{eq:Km-quasi}
	K_m(s)
	\leq\max\left\{0,\sup_{0<\eta<1}
	f(\eta)\right\}
	=\max\left\{m-s, \frac{m(\dim_{qA} E-s)}{\dim_{qA} E}\right\}.
\end{align}

Moreover, it is easy to see that $K_m$ is $1$-Lipschitz and hence continuous. 

 \begin{thmm}\label{lowerbound-spectrum}
    Let $E\subset\R^d$ be non-empty and bounded. Let $m>0$, and let  $K_m$ be defined as in Definition~\ref{def:Km}. Then we have
    \begin{equation}\label{lowerbound-spectrum-upper}
        \ubd^m E\geq\sup_{\theta\in(0,1]}
        \min\left\{\uid E,
        \frac{m\uid E}{\uid E+\theta K_m(\uid E)}\right\}
    \end{equation}
    and
    \begin{equation}\label{lowerbound-spectrum-lower}
        \lbd^m E\geq\sup_{\theta\in(0,1]}
        \min\left\{\lid E,
        \frac{m\lid E}{\lid E+\theta K_m(\lid E)}\right\}.
    \end{equation}
 \end{thmm}

 \begin{proof}
    We only prove \eqref{lowerbound-spectrum-upper}; the proof of \eqref{lowerbound-spectrum-lower} is similar and is omitted. Fix $\theta\in(0,1]$. Without loss of generality, we assume
    $\uid E>0$ and let $0<s<\uid E$. According to Lemma
    \ref{Frostman}, there exists $c>0$ such that for arbitrarily small
    $0<r<1$ we can find $\mu_r\in\mP(E)$ satisfying
    \begin{equation}\label{spec-frostman}
        \mu_r(B(x,h))\leq c h^s
        \quad\text{for all $x\in\R^d$ and $r\leq h\leq r^\theta$.}
    \end{equation}
    Fix $\epsilon>0$, and let $n$ be an integer large enough such that $\frac{m}{n}<\frac{\epsilon}{2}$.
    For $1\leq j\leq n-1$, we write
    \begin{equation*}
        t_j=\overline{\dim}_A^{\frac{j}{n}}E+\frac{\epsilon}{2}.
    \end{equation*}
    According to the definition of the upper Assouad spectrum \eqref{defi-upperassouad}, there exists $c_j>0$ such that
    \begin{equation}\label{Nr}
        N_{r^\theta}(B(x,r^{\frac{\theta j}{n}})\cap E)
        \leq c_j\left(\frac{r^{\frac{\theta j}{n}}}{r^\theta}\right)^{t_j}
        \quad\text{for all $x\in E$ and $0<r<1$.}
    \end{equation}
    It follows from \eqref{spec-frostman} and \eqref{Nr} that
    \begin{equation*}
        \mu_r(B(x,r^{\frac{\theta j}{n}}))
        \leq\min\left\{1,c\,c_jr^{\theta(s-(1-\frac{j}{n})t_j)}\right\}
        \leq\max\{1,c\,c_j\}r^{\theta(s-(1-\frac{j}{n})t_j)_+}.
    \end{equation*}
    Then for all $r^{\frac{\theta(j+1)}{n}}\leq h\leq r^{\frac{\theta j}{n}}$, we have 
    \begin{align*}
    	\mu_r(B(x,h))h^{-m}
    	&\leq  \mu_r(B(x,r^{\frac{\theta j}{n}}))r^{-\frac{\theta(j+1)m}{n}}\\        
    	&\leq\max\{1,c\,c_j\}
    	r^{\theta(-\frac{m(j+1)}{n}+(s-(1-\frac{j}{n})t_j)_+)}\\
    	&\leq\max\{1,c\,c_j\}r^{-\theta(K_m(s)+\epsilon)},
    \end{align*}
    where the last inequality follows from
    \begin{align*}
    K_m(s)&\geq m\cdot \frac{j}{n}-\left(s-(1-\frac{j}{n})(t_j-\frac{\epsilon}{2})\right)_+\\
    &\geq m\cdot \frac{j}{n}-\left(s-(1-\frac{j}{n})t_j\right)_+-\frac{\epsilon}{2}\\
    &\geq  \frac{m(j+1)}{n}-\left(s-(1-\frac{j}{n})t_j\right)_+-\epsilon.
    \end{align*}
    For $r^{\frac{\theta}{n}}\leq h\leq1$, we have
    \begin{equation*}
        \mu_r(B(x,h))h^{-m}\leq h^{-m}
        \leq r^{-\frac{\theta m}{n}}\leq r^{-\theta(K_m(s)+\epsilon)}.
    \end{equation*}
    Thus we conclude that
 \begin{equation*}
	\mu_r(B(x,h))h^{-m}\leq c_0r^{-\theta(K_m(s)+\epsilon)}
	\quad\text{for all $x\in E$ and $r^\theta\leq h\leq1$.}
\end{equation*}    
    Here we take
     $c_0=\max\{1,c\,c_1,\ldots,c\,c_{n-1}\}$.
     Let
     \begin{equation*}
     	a=\max\left\{1,\frac{s+\theta(K_m(s)+\epsilon)}{m}\right\}.
     \end{equation*}
    Then
       \begin{align}\label{th3-eq1}
     mr^{am}	\int_{r^\theta}^{1}
     	\mu_r(B(x,h))h^{-m-1}\dih\leq c_0mr^{am-\theta(K_m(s)+\epsilon)} \cdot \log\frac{1}{r^\theta}\leq c_0mr^s\cdot\log\frac{1}{r^\theta}.
     \end{align}
    For $r^a\leq h\leq r$, it follows from \eqref{spec-frostman} that
    \begin{align}\label{th3-eq2}
    mr^{am}\int_{r^a}^r
    \mu_r(B(x,h))h^{-m-1}\dih\leq cmr^{am}\int_{r^a}^r
    r^sh^{-m-1}\dih\leq cr^s.
    \end{align}
     For $r\leq h\leq r^\theta$, it is easy to see that
     \begin{align*}
      r^{am}h^{s-m}\leq&\begin{cases}
      	r^{am+s-m}, \quad &s\leq m;\\
      	r^{am}\quad &s>m;
      \end{cases} \\
      \leq& r^s.
     \end{align*}
    It follows from \eqref{spec-frostman} that
      \begin{align}\label{th3-eq3}
    	mr^{am}\int_{r}^{r^{\theta}}
    	\mu_r(B(x,h))h^{-m-1}\dih\leq cmr^{am}\int_{r}^{r^{\theta}}
    h^{s-m-1}\dih\leq cm r^s\cdot \log(r^{\theta-1}).
    \end{align}
    Moreover,
     \begin{align}\label{th3-eq4}
    	mr^{am}\int_{1}^{\infty}
    	\mu_r(B(x,h))h^{-m-1}\dih\leq r^{am}\leq r^s.
    \end{align}
    According to \eqref{th3-eq1}-\eqref{th3-eq4}, we have
    for all $x\in E$,
    \begin{align*}
        \phi_{r^a}^m\ast\mu_r(x)
        =mr^{am}\int_{r^a}^\infty
        \mu_r(B(x,h))h^{-m-1}\dih
        \lesssim r^s\cdot(- \log r).
    \end{align*}
   Consequently,
    \begin{equation*}
        \ubd^m E\geq\frac{s}{a}
        =\min\left\{s,\frac{ms}{s+\theta(K_m(s)+\epsilon)}\right\}.
    \end{equation*}
    Letting $\epsilon$ tend to $0$ and $s$ tend to $\uid E$ gives
    \begin{equation*}
        \ubd^m E\geq
        \min\left\{\uid E,\frac{m\uid E}{\uid E+\theta K_m(\uid E)}\right\}.
    \end{equation*}
     By taking the supremum over $\theta\in(0,1]$, we  prove
    \eqref{lowerbound-spectrum-upper}.
 \end{proof}

To obtain a simpler bound and clarify the relationship between
box dimension profiles and the box dimension, the upper Assouad
spectrum, and the Assouad dimension, we set $\theta=1$ in
Theorem \ref{lowerbound-spectrum} and obtain the following corollary.

\begin{cor}\label{cor:assouad-spectrum-bound}
    Let $E\subset\R^d$ be non-empty and bounded, and let $m>0$ and $0<\xi<1$. Write
    \begin{equation*}
        \Delta_\xi=\max\left\{0,\overline{\dim}_A^\xi E-m,
        (\dim_A E-m)(1-\xi)\right\}.
    \end{equation*}
    Then
    \begin{equation*}
        \ubd^m E\geq \frac{m\ubd E}{m+\Delta_\xi}\quad\text{ and }\quad \lbd^m E\geq \frac{m\lbd E}{m+\Delta_\xi}.
    \end{equation*}
\end{cor}

\begin{proof}
    Since the upper Assouad spectrum is non-decreasing and is bounded above by $\dim_A E$, we have for all $0<\eta<1$,
    \begin{align*}
    	 (1-\eta)(\overline{\dim}_A^\eta E-m)&\leq
    	 \begin{cases}
    	 	0, \quad&\text{if }\overline{\dim}_A^\eta E\leq m;\\
    	 	\overline{\dim}_A^\xi E-m	, \quad&\text{if }\overline{\dim}_A^\eta E> m \text{ and } \eta\leq \xi;\\
    	 (1-\xi)({\dim}_A  E-m)	, \quad&\text{if }\overline{\dim}_A^\eta E> m\text{ and } \eta> \xi;
    	 \end{cases}\\
    	&\leq\Delta_\xi.
    \end{align*}
  Then
    \begin{align}\label{remark1eq1}
        K_m(\ubd E)
        &\leq\max\left\{0,m-\ubd E+
        \sup_{0<\eta<1}(1-\eta)(\overline{\dim}_A^\eta E-m)\right\}\notag\\
        &\leq m-\ubd E+\Delta_\xi,
    \end{align}
    where the last inequality follows from
    \begin{equation*}
        m-\ubd E+\Delta_\xi
        \geq\overline{\dim}_A^\xi E-\ubd E\geq0.
    \end{equation*}
    Applying Theorem \ref{lowerbound-spectrum} at $\theta=1$ and using \eqref{remark1eq1}, we obtain
    \begin{align*}
        \ubd^m E
       \geq&\min\left\{\ubd E,
        \frac{m\ubd E}{\ubd E+K_m(\ubd E)}\right\}\\
        \geq&\min\left\{\ubd E,
        \frac{m\ubd E}{m+\Delta_\xi}\right\}\\
        =&\frac{m\ubd E}{m+\Delta_\xi}.
    \end{align*}
    
    The lower box dimension case follows by the same argument, using the lower version of Theorem \ref{lowerbound-spectrum}.
\end{proof}

\begin{remark}
    For $0<m\leq d$, Falconer, Fraser, and Shmerkin showed in \cite[Theorem~2.2]{FFS} that
    \begin{equation}\label{FFSbound}
        \ubd^m E\geq\ubd E-\Delta_\xi.
    \end{equation}
    Since
       \begin{equation*}
    	\frac{m\ubd E}{m+\Delta_\xi}-(\ubd E-\Delta_\xi)
    	=\frac{\Delta_\xi(m+\Delta_\xi-\ubd E)}{m+\Delta_\xi}\geq0,
    \end{equation*}
    Corollary \ref{cor:assouad-spectrum-bound} gives  a lower bound that is at least as strong as \eqref{FFSbound}.
   As a consequence of \eqref{FFSbound}, it was also shown in \cite{FFS} that $\ubd^m E=\ubd E$ if $\dim_{qA} E\leq m$.
\end{remark}

 Bounding the upper Assouad spectrum by the quasi-Assouad dimension
 yields the following result.

\begin{cor}
 	Let $E$ be a non-empty and bounded set in $\R^d$. If $\dim_{qA} E>m$, then
 	\begin{equation}\label{eq:quasi-threshold}
 		\ubd^m E\geq
 		\sup\left\{
 		\uid E:
 		0<\theta\leq1,\quad
 		\uid E\leq
 		\frac{m(1-\theta)\dim_{qA} E}
 		{\dim_{qA} E-m\theta}
 		\right\}.
 	\end{equation}
    The same conclusion holds with $\ubd^m E$ and $\uid E$ replaced by $\lbd^m E$ and $\lid E$, respectively. 
 \end{cor}

 \begin{proof}
 	We only prove \eqref{eq:quasi-threshold}.
 It follows from \eqref{eq:Km-quasi} that for  all $\theta\in(0,1]$,  
    \begin{equation*}
        \frac{ms}{s+\theta K_m(s)}
        \geq\frac{ms\dim_{qA} E}{s\dim_{qA} E+m\theta(\dim_{qA} E-s)}.
    \end{equation*}
     It is easy to see that
      \begin{equation*}
     	\frac{ms\dim_{qA} E}{s\dim_{qA} E+m\theta(\dim_{qA} E-s)}\geq s
     	\quad\Longleftrightarrow\quad
     	s\leq\frac{m\dim_{qA} E(1-\theta)}{\dim_{qA} E-m\theta}.
     \end{equation*}
   Thus we conclude \eqref{eq:quasi-threshold}.
 \end{proof}

 \begin{thmm}\label{thm:initial-growth-profiles}
 	Let $E\subset\R^d$ be a non-empty and bounded set with positive upper box dimension.  Let
 	\[
 	\kappa:=\liminf_{\theta\downarrow0}\frac{\uid E}{\theta}
 	\quad\text{and}\quad
 	N:=\frac{\kappa\ubd E}{\kappa-\ubd E}.
 	\]
 	If $N\geq\dim_{qA} E$, then 
 	\begin{equation}\label{eq:initial-growth-exact}
 		\ubd^m E
 		=\inf_{\theta\in(0,1]}
 		\frac{m\uid E}{m\theta+(1-\theta)\uid E}
 		=\min\left\{\ubd E,\frac{m\kappa}{m+\kappa}\right\}.
 	\end{equation}
 	Here $\frac{\kappa\ubd E}{\kappa-\ubd E}$ and $\frac{m\kappa}{m+\kappa}$ are  interpreted as $\ubd E$ and  $m$ respectively when
 	$\kappa=\infty$.
 	The same conclusions hold with $\uid E$, $\ubd E$, and
 	$\ubd^m E$ replaced throughout by $\lid E$, $\lbd E$, and
 	$\lbd^m E$, respectively, with $\kappa$ and $N$ defined
 	using the corresponding lower dimensions.
 \end{thmm}

 \begin{proof}
    We only prove 
    \eqref{eq:initial-growth-exact}. By the definition of $\kappa$, there exists a sequence
    $\{\theta_j\}_j$ tending to $0$ such that
    \begin{equation*}
        \frac{\overline{\dim}_{\theta_j}E}{\theta_j}\longrightarrow\kappa.
    \end{equation*}
    It follows that
    \begin{equation*}
        \frac{m\overline{\dim}_{\theta_j}E}
        {m\theta_j+(1-\theta_j)\overline{\dim}_{\theta_j}E}
        \longrightarrow\frac{m\kappa}{m+\kappa} \quad \text{as $j$ tends to $\infty$}.
    \end{equation*}
     According to  Theorem \ref{upperbound2}, we have   
    \begin{equation}\label{eq:initial-growth-upper}
        \ubd^m E
        \leq\inf_{\theta\in(0,1]}
        \frac{m\uid E}{m\theta+(1-\theta)\uid E}
        \leq\min\left\{\ubd E,\frac{m\kappa}{m+\kappa}\right\}.
    \end{equation}

    Suppose first that $0<m<N$. 
    It follows from 
     \eqref{eq:Km-quasi} that 
     \begin{align*}
     K_m(\ubd E)\leq&\max\left\{\frac{m(\dim_{qA} E-\ubd E)}{\dim_{qA} E},m-\ubd E\right\}\\
     \leq&\frac{m(N-\ubd E)}{N}\quad\text{(by $\dim_{qA}E\leq N$)}.
     \end{align*}
    Applying Theorem
    \ref{lowerbound-spectrum} at $\theta=1$, we therefore obtain
    \begin{align}\label{eq:initial-growth-lower}
        \ubd^m E
        &\geq\min\left\{\ubd E,\frac{m\ubd E}{\ubd E+K_m(\ubd E)}\right\}\notag\\
        &\geq\min\left\{\ubd E,\frac{mN\ubd E}{N\ubd E+m(N-\ubd E)}\right\}\notag\\
        &=\frac{mN\ubd E}{mN+(N-m)\ubd E}\notag\\
        &=\frac{m\kappa}{m+\kappa}.
    \end{align}
    Now we assume  $m\geq N$. Let  $0<k<N\leq m$. Then it follows from \eqref{eq:initial-growth-lower} that
    \begin{equation*}
        \ubd^m E\geq\ubd^k E
        \geq\frac{kN\ubd E}{kN+(N-k)\ubd E}.
    \end{equation*}
    Letting $k\nearrow N$ yields 
    \begin{equation}\label{th3eq1}
    	\ubd^m E\geq \ubd E=\min\left\{\ubd E,\frac{m\kappa}{m+\kappa}\right\}.
    \end{equation}
    Combining \eqref{eq:initial-growth-upper}-\eqref{th3eq1} gives
    \eqref{eq:initial-growth-exact}.
 \end{proof}
 
 \begin{remark}
 	 If $\kappa=0$, then it follows from \eqref{eq:initial-growth-upper} that
 	$\ubd^m E=0$ for all $m>0$.
 	Moreover, under the hypotheses of Theorem \ref{thm:initial-growth-profiles},
 	if $\kappa<\infty$, then 
 	\[
 	\ubd^m E=\frac{m\kappa}{m+\kappa}<\min\{m,\ubd E\}
 	\qquad\text{whenever }0<m<N.
 	\]
 \end{remark}

\begin{remark}\label{rem:initial-growth-hypothesis}
    The hypothesis $N\geq\dim_{qA}E$ in
    Theorem \ref{thm:initial-growth-profiles} cannot in general be omitted. Indeed, let $ E=A\times F_p$ be the set
    from Example \ref{ex}. Then for $0<\theta\leq1$, we have
    \[
       \dim_\theta E=c+\frac{\theta}{p+\theta}
        \quad \text{and} \quad
         \dim_B E=c+\frac{1}{p+1}.
    \]
    Thus 
    $\kappa=\infty$ and $N= \dim_B E$. If the hypothesis were omitted,
    Theorem \ref{thm:initial-growth-profiles} would give
    $$\ubd^1 E=\min\{1,\dim_B E\},$$which contradicts the strict inequality in Example \ref{ex}.
\end{remark}

	We now present two examples in which Theorem
	\ref{thm:initial-growth-profiles} yields explicit formulas
	for box dimension profiles. 
\begin{ex}\label{ex:inverted-lattice-profiles}
    Let $p>0$, $d\in\mathbb N$ and $0<m< d$. Let
    \begin{equation*}
        G_{p,d}:=\left\{\frac{x}{\|x\|^2}:
        x\in\{1^p,2^p,3^p,\ldots\}^d\right\}
    \end{equation*}
    be the inversion of the lattice $\{1^p,2^p,3^p,\ldots\}^d$ in the unit sphere in $\R^d$.
    By \cite[Proposition 3.8]{BF}, for every $0<\theta\leq1$,
    \begin{equation*}
        \uid G_{p,d}=\lid G_{p,d}=\frac{d\theta}{p+\theta}.
    \end{equation*}
    Consequently,
    \begin{equation*}
        \dim_B  G_{p,d}=\frac{d}{p+1},
        \qquad
        \kappa=\lim_{\theta\downarrow0}\frac{\dim_\theta G_{p,d}}{\theta}
        =\frac{d}{p}.
    \end{equation*}
    It follows that
    \begin{equation*}
        N=\frac{\kappa\dim_B G_{p,d}}{\kappa-\dim_B G_{p,d}}
        =d\geq\dim_{qA}G_{p,d}.
    \end{equation*}
    According to Theorem \ref{thm:initial-growth-profiles}, we have
\begin{equation*}
        \dim_B^m G_{p,d}
        =\min\left\{\frac{d}{p+1},\frac{dm}{d+pm}\right\}
        =\frac{dm}{d+pm}<\min\{m,  \dim_B  G_{p,d}\}.
    \end{equation*}
\end{ex}

 Example
 \ref{ex:inverted-lattice-profiles} exhibits a dimension drop,
 whereas  the following example \ref{ex:product-geometric-profiles}
 illustrates dimension preservation.

\begin{ex}\label{ex:product-geometric-profiles}
    Let $p>0$ and let
    \begin{equation*}
        F_p=\{0\}\cup\{n^{-p}:n\in\mathbb N\},
        \qquad
        G=\{0\}\cup\{2^{-n}:n\in\mathbb N\},
        \qquad E=G\times F_p.
    \end{equation*}
    For $0<r<R<1$, it is not difficult to see that
    \begin{equation*}
      \sup_xN_r(B(x,R)\cap G)\lesssim 1+\log\frac{R}{r}.
    \end{equation*}
    Thus  $\dim_A G=0$.
 An argument similar to \eqref{ex1eq1} gives that
   $$\dim_\theta E   =\frac{\theta}{p+\theta}.$$
    Thus,
    \begin{equation*}
        \dim_B E=\frac{1}{p+1},
        \qquad
        \kappa=\lim_{\theta\downarrow0}\frac{\dim_\theta E}{\theta}
        =\frac{1}{p},
        \qquad
        N=\frac{\kappa \dim_B E}{\kappa- \dim_B E}=1.
    \end{equation*}
    {
    By \cite[Theorem~2.1]{Fr14},  we have
    \begin{equation*}
        \dim_{A} F_p  \leq \dim_A E
        \leq\dim_A F_p+\dim_A G=\dim_A F_p.
    \end{equation*}
    }
    {Thus  $\dim_{qA}E\leq\dim_A E=\dim_A F_p\leq1=N$. Hence}
    according to Theorem \ref{thm:initial-growth-profiles} that 
\begin{equation*}
        \lbd^1 E=\ubd^1 E=\frac{1}{p+1}=\dim_B E.
    \end{equation*}
\end{ex}

\noindent  \textbf{Acknowledgements:} 
  Most of this work was completed while the author was a postdoc at Fudan University. The author acknowledges financial support from Professor Weixiao Shen's research grant, including the New Cornerstone Science Foundation through the New Cornerstone Investigator Program.
  
  \noindent  \textbf{Declaration of AI Use:} ChatGPT was used to assist with Section \ref{sec:proof}, refining some of the conclusions, as well as for language editing and identifying typographical errors throughout the manuscript. 
  More specifically, while the upper bound estimate in Proposition \ref{thm:upper} was derived independently by the authors,   ChatGPT pointed out that this bound was suboptimal (see Remark \ref{remark1}). Consequently, Theorems \ref{upperbound2} and \ref{thm:initial-growth-profiles} were formulated with the assistance of ChatGPT.
  The mathematical proofs in Sections \ref{sec-intermediate} and \ref{sec-profile} were completed independently by the authors.

\bibliographystyle{plain} 

\end{document}